\documentclass{article}
\usepackage{amsmath,amssymb,amsthm,fullpage}
\usepackage{array,graphicx}

\renewcommand{\deg}{\operatorname{deg}}

\newcommand{\LHS}{\operatorname{LHS}}

\newcommand{\RHS}{\operatorname{RHS}}

\newtheorem{thm}{Theorem}[section]
\newtheorem{lem}[thm]{Lemma}
\newtheorem{cor}[thm]{Corollary}
\newtheorem{obs}[thm]{Observation}

\theoremstyle{definition}

\newtheorem*{exa}{Example}

\title{Adjacency relationships forced by a degree sequence}
\author{Michael D. Barrus\\
\small Department of Mathematics\\
\small University of Rhode Island\\
\small Kingston, RI 02881\\
\small \tt barrus@uri.edu}

\begin{document}
\maketitle
\begin{abstract}
There are typically several nonisomorphic graphs having a given degree sequence, and for any two degree sequence terms it is often possible to find a realization in which the corresponding vertices are adjacent and one in which they are not. We provide necessary and sufficient conditions for two vertices to be adjacent (or nonadjacent) in every realization of the degree sequence. These conditions generalize degree sequence and structural characterizations of the threshold graphs, in which every adjacency relationship is forcibly determined by the degree sequence. We further show that degree sequences for which adjacency relationships are forced form an upward-closed set in the dominance order on graphic partitions of an even integer.
\end{abstract}

\section{Introduction}\label{sec: intro}

A fundamental goal of the study of graph degree sequences is to identify properties that must be shared by all graphs having the same degree sequence. In this paper we address one of the simplest of graph properties: whether two given vertices are adjacent.

Most degree sequences $d$ are shared by multiple distinct graphs. We call these graphs the \emph{realizations} of $d$. In this paper we consider only labeled graphs, that is, we distinguish between realizations having distinct edge sets, even if these realizations are isomorphic. Throughout the paper we will consider a degree sequence $d=(d_1,\dots,d_n)$ and all realizations of $d$ with vertex set $V = \{1,\dots,n\}$ (we denote such a range of natural numbers by $[n]$) that satisfy the condition that each vertex $i$ has the corresponding degree $d_i$. We will assume in each case, unless otherwise stated, that $d_1 \geq \dots \geq d_n$.

For only one type of degree sequence are all the adjacency relationships in a realization completely determined. These are the \emph{threshold sequences}, those sequences having only one realization. \emph{Threshold graphs}, the graphs realizing threshold sequences, were introduced (via an equivalent definition) by Chv\'{a}tal and Hammer in \cite{ChvatalHammer73,ChvatalHammer77}, as well as by many other authors independently. These graphs have a number of remarkable properties; see the monograph~\cite{MahadevPeled95} for a survey and bibliography. We will refer to several of these properties in the course of the paper.

On the other end of the spectrum from the threshold sequences, many degree sequences have the property that \emph{any} fixed pair of vertices may be adjacent in one realization and nonadjacent in another; such is the case, for example, with $(1,1,1,1)$. For still other sequences, some adjacency relationships are determined while others are not; notice that in the two realizations of $(2,2,1,1,0)$ the vertices of degree $2$ must be adjacent, the vertices of degree $1$ must be nonadjacent, and the vertex of degree $0$ cannot be adjacent to anything, while a fixed vertex of degree $1$ may or may not be adjacent to a fixed vertex of degree $2$.

Suppose that $d$ is an arbitrary degree sequence. We classify pairs $\{i,j\}$ of vertices from $V$ as follows. If $i$ and $j$ are adjacent in every realization of $d$, we say that $\{i,j\}$ is a \emph{forced edge}. If $i$ and $j$ are adjacent in no realization of $d$, then $\{i,j\}$ is a \emph{forced non-edge}. Vertices in a forced edge or forced non-edge are \emph{forcibly adjacent} or \emph{forcibly nonadjacent}, respectively. If $\{i,j\}$ is either a forced edge or forced non-edge, we call it a \emph{forced pair}; otherwise, it is \emph{unforced}. By definition, in threshold graphs every pair of vertices is forced.

In this paper we characterize the forced pairs for general degree sequences. We present conditions that allow us to recognize these pairs from the degree sequence and describe the structure they as a set create in any realization of the degree sequence.

As an alternative viewpoint, given a degree sequence $d$, we may define the \emph{intersection envelope graph} $I(d)$ (respectively, \emph{union envelope graph $U(d)$}) to be the graph with vertex set $[n]$ whose edge set is the intersection (resp., union) of the edge sets of all realizations of $d$. The forced edges of $d$ are precisely the edges of $I(d)$, and the forced non-edges of $d$ are precisely the non-edges of $U(d)$. As we will see, $I(d)$ and $U(d)$ are threshold graphs, and our results allow us to describe these graphs.

One particular property of threshold sequences contextualized by a study of forced pairs is the location of these sequences in the dominance (majorization) order on degree sequences having the same sum. Threshold sequences comprise the maximal elements in this order, and we show that as a collection, degree sequences with forced pairs majorize degree sequences having no forced pairs.

The structure of the paper is as follows: In Section 2 we provide necessary and sufficient conditions on a degree sequence for a pair $\{i,j\}$ to be a forced edge or forced non-edge among realizations of a degree sequence $d$. We then give an alternative degree sequence characterization in terms of Erd\H{o}s--Gallai differences, which we introduce. In Section 3 we study the overall structure of forced pairs in a graph, describing the envelope graphs $I(d)$ and $U(d)$. Finally, in Section 4 we present properties of forced pairs in the context of the dominance order on degree sequences.

Throughout the paper all graphs are assumed to be simple and finite. We use $V(H)$ to denote the vertex set of a graph $H$. A list of nonnegative integers is \emph{graphic} if it is the degree sequence of some graph. A \emph{clique} (respectively, \emph{independent set}) is a set of pairwise adjacent (nonadjacent) vertices.

\section{Degree sequence conditions for forced pairs} \label{sec: degree conditions}
We begin with a straightforward test for determining whether a pair of vertices is forced.

\begin{thm} \label{thm: forced iff not graphic}
Given the degree sequence $d=(d_1,\dots,d_n)$ and vertex set $[n]$, let $i,j$ be distinct elements of $[n]$ such that $i<j$. The pair $\{i,j\}$ is a forced edge if and only if the sequence \[d^+(i,j) = (d_1,\dots,d_{i-1},d_i + 1, d_{i+1}, \dots, d_{j-1}, d_j + 1, d_{j+1},\dots, d_n)\] is not graphic. The pair $\{i,j\}$ is a forced non-edge if and only if the sequence \[d^-(i,j) = (d_1,\dots,d_{i-1},d_i - 1, d_{i+1}, \dots, d_{j-1}, d_j - 1, d_{j+1},\dots, d_n)\] is not graphic.
\end{thm}

Before proving this theorem, we introduce some notation. Given a degree sequence $\pi$ of length $n$, let $\overline{\pi}$ denote the degree sequence of the complement of a realization of $\pi$, i.e., $\overline{\pi} = (n-1-d_n, \dots, n-1-d_1)$; we call $\overline{\pi}$ the complementary degree sequence of $\pi$. Note that $\pi$ is also the complementary degree sequence of $\overline{\pi}$.

\begin{proof}
We begin by proving the contrapositives of the statements in the first equivalence. Suppose first that $\{i,j\}$ is not a forced edge for $d$. There must exist a realization $G$ of $d$ in which $i$ and $j$ are not adjacent. The graph $H$ formed by adding edge $ij$ to $G$ has degree sequence $d^+(i,j)$, so $d^+(i,j)$ is graphic.

Suppose now that $d^+(i,j)$ is graphic, and let $H$ be a realization. Suppose also that $G$ is a realization of $d$. If $\{i,j\}$ is not an edge of $G$, then it is not a forced edge for $d$. Furthermore, if $\{i,j\}$ is an edge of $H$, then removing that edge produces a realization of $d$ with no edge between $i$ and $j$, so once again $\{i,j\}$ is not a forced edge. Suppose now that $ij$ is an edge of $G$ but not of $H$. Let $J$ be the symmetric difference of $G$ and $H$, that is, the graph on $[n]$ having as its edges all edges belonging to exactly one of $G$ and $H$. Color each edge in $J$ red if it is an edge of $G$ and blue if it is an edge of $H$. Since the degree of any vertex in $[n]$ other than $i$ and $j$ is the same in both $G$ and $H$, there is an equal number of red and blue edges meeting at such a vertex. For all such vertices, partition the incident edges into pairs that each contain a red and a blue edge. Now vertices $i$ and $j$ each are incident with one more blue edge than red; fix a vertex $v$ such that $iv$ is a blue edge in $J$ and partition the other edges incident with $i$ into pairs containing a red and a blue edge. Do the same thing for the edges incident with $j$ other than a fixed blue edge $jw$. We now find a path from $i$ to $j$ in $J$ whose edges alternate between blue and red. Note that $iv$ is a blue edge, and that this edge is paired with a red edge incident with $v$, which is in turn paired with a blue edge at its other endpoint, and so on. Since each edge of $J$ other than $iv$ and $jw$ is paired with a unique edge of the opposite color at each of its endpoints, the path beginning with $iv$ must continue until it terminates with edge $wj$. Now let $v_0,v_1,\dots,v_\ell$ be the vertices encountered on this path, in order, so that $v_0=i$, $v_1=v$, $v_{\ell-1} = w$, and $v_\ell = j$. The graph $G$ contains edges $v_1v_2, v_3v_4, \dots, v_{\ell-2}v_{\ell-1}$ and $v_\ell v_0$ and non-edges $v_0v_1, v_2v_3, \dots, v_{\ell-1}v_\ell$. Deleting these edges and adding the non-edges as new edges creates a realization of $d$ where $i$ and $j$ are not adjacent, so once again $\{i,j\}$ is not a forced edge for $d$.

Since $\{i,j\}$ is an edge in a realization of $\pi$ if and only if it is not an edge in a realization of $\overline{\pi}$, the pair $\{i,j\}$ is a forced non-edge for $d$ if and only if it is a forced edge for $\overline{d}$, which is equivalent by the preceding paragraph to having $\overline{d}^{+}(i,j)$ not be graphic. Since a list $\pi$ of integers is a degree sequence if and only if $\overline{\pi}$ is a degree sequence, and we can easily verify that $d^-(i,j) = \overline{\overline{d}^+(i,j)}$, the pair $\{i,j\}$ is a forced non-edge in $G$ if and only if $d^{-}(i,j)$ is not a graphic sequence.
\end{proof}

By combining Theorem~\ref{thm: forced iff not graphic} with a test for graphicality we may find alternate characterizations of forced pairs. Here we will use the well known Erd\H{o}s--Gallai criteria~\cite{ErdosGallai60} with a simplification due to Hammer, Ibaraki, and Simeone~\cite{HammerIbarakiSimeone78,HammerIbarakiSimeone81}). For any integer sequence $\pi=(\pi_1,\dots,\pi_n)$, define $m(\pi) = \max\{i:\pi_i \geq i-1\}$.

\begin{thm}[\cite{ErdosGallai60,HammerIbarakiSimeone78,HammerIbarakiSimeone81}] \label{thm: ErdosGallai}
A list $\pi=(\pi_1,\dots,\pi_n)$ of nonnegative integers in descending order is graphic if and only if $\sum_k \pi_k$ is even and \[\sum_{\ell \leq k} \pi_\ell \leq k(k-1) + \sum_{\ell>k} \min\{k, d_\ell\}\] for each $k \in \{1,\dots,m(\pi)\}$.
\end{thm}

For each $k \in [n]$, let $\LHS_k(\pi) = \sum_{\ell \leq k} \pi_\ell$ and $\RHS_k(\pi) = k(k-1) + \sum_{\ell>k} \min\{k, d_\ell\}$. We now define the \emph{$k$th Erd\H{o}s--Gallai difference} $\Delta_k(\pi)$ by \[\Delta_k(\pi) = \RHS_k(\pi) - \LHS_k(\pi).\] Note that an integer sequence with even sum is graphic if and only if these differences are all nonnegative.

\begin{thm} \label{thm: forced via EG diff}
Let $d=(d_1,\dots,d_n)$ be a graphic list, and let $i,j$ be integers such that $1 \leq i<j \leq n$. The pair $\{i,j\}$ is a forced edge for $d$ if and only if there exists $k$ such that $1 \leq k \leq n$ and one of the following is true:
\begin{enumerate}
\item[\textup{(1)}] $\Delta_k(d) \leq 1$ and $j \leq k$.
\item[\textup{(2)}] $\Delta_k(d) = 0$; $i \leq k < j$; and $k \leq d_j$.
\end{enumerate}
The pair $\{i,j\}$ is a forced non-edge for $d$ if and only if there exists $k$ such that $1 \leq k \leq n$ and one of the following is true:
\begin{enumerate}
\item[\textup{(3)}] $\Delta_k(d) \leq 1$ and $d_i < k < i$.
\item[\textup{(4)}] $\Delta_k(d) = 0$; $k<i$; and $d_j \leq k \leq d_i$.
\end{enumerate}
\end{thm}
\begin{proof}
By Theorems~\ref{thm: forced iff not graphic} and~\ref{thm: ErdosGallai}, $\{i,j\}$ is a forced edge if and only if there exists an integer $k$ such that $\Delta_k\left(d^+(i,j)\right) < 0$. We prove that this happens if and only if condition (1) or condition (2) holds. Let $k$ be an arbitrary element of $[n]$.

\medskip
\emph{Case: $k<i$.} In this case neither condition (1) nor condition (2) holds. Furthermore, $\LHS_k(d^+(i,j)) = \LHS_k(d) \leq \RHS_k(d) \leq \RHS_k(d^+(i,j))$, so $\Delta_k\left(d^+(i,j)\right) \geq 0$.

\medskip
\emph{Case: $j \leq k$.} Here condition (2) does not hold. Since $\RHS_k(d^+(i,j))=\RHS_k(d)$ and $\LHS_k(d^+(i,j)) \leq \LHS_k(d)+2$, we see that $\Delta_k\left(d^+(i,j)\right) < 0$ if and only if $\Delta_k(d) \leq 1$, which is equivalent to condition (1).

\medskip
\emph{Case: $i \leq k < j$.} Note that condition (1) cannot hold in this case. Since $i \leq k$, we have $\LHS_k(d^+(i,j))=\LHS_k(d)+1$. If $\Delta_k(d) \geq 1$, then $\Delta_k(d^+(i,j)) \geq 0$. If $d_j<k$, then $\RHS_k(d^+(i,j)) = \RHS_k(d)+1$ and $\Delta_k(d^+(i,j)) \geq 0$. If $\Delta_k(d)=0$ and $d_j \geq k$, then $\RHS_k(d^+(i,j))=\RHS_k(d)$ and hence $\Delta_k\left(d^+(i,j)\right) = -1$. Hence $\Delta_k\left(d^+(i,j)\right) < 0$ is equivalent to condition (2).

\medskip
We now consider forced non-edges of $d$. By Theorem~\ref{thm: forced iff not graphic}, $\{i,j\}$ is a forced edge if and only if there exists an integer $k$ such that $\Delta_k\left(d^-(i,j)\right) < 0$. We show that this happens if and only if condition (3) or condition (4) holds. Let $k$ be an arbitrary element of $[n]$. Note that if $k \geq i$ then neither (3) nor (4) holds, and $\LHS_k(d^-(i,j)) < \LHS_k(d)$, forcing $\Delta_k(d^-(i,j)) \geq 0$. Assume now that $k < i$. This forces $\LHS_k(d^-(i,j))=\LHS_k(d)$.

\medskip
\emph{Case: $k\leq d_j$.} Neither condition (3) nor condition (4) holds, and $\RHS_k(d^-(i,j)) = \RHS_k(d)$, so $\Delta_k\left(d^-(i,j)\right) \geq 0$.

\medskip
\emph{Case: $d_i < k$.} Here condition (4) fails. Since $\RHS_k(d^-(i,j))=\RHS_k(d)-2$, we see that $\Delta_k\left(d^-(i,j)\right) < 0$ if and only if $\Delta_k(d) \leq 1$, which is equivalent to condition (3).

\medskip
\emph{Case: $d_j \leq k \leq d_i$.} Here condition (3) fails. Since $d_j\leq k$, we have $\RHS_k(d^-(i,j))=\RHS_k(d)-1$. If $\Delta_k(d) \geq 1$, then $\Delta_k(d^-(i,j)) \geq 0$. If $\Delta_k(d) = 0$, then $\Delta_k\left(d^+(i,j)\right) = -1$. Hence $\Delta_k\left(d^+(i,j)\right) < 0$ is equivalent to condition (4).
\end{proof}

\section{Structure induced by forced pairs}

Theorems~\ref{thm: forced iff not graphic} and~\ref{thm: forced via EG diff} allow us to determine if a single pair of vertices comprises a forced edge or forced non-edge by examining the degree sequence. In this section we determine the structure of all forcible adjacency relationships by describing the envelope graphs $I(d)$ and $U(d)$ introduced in Section~\ref{sec: intro}.

Recall that the edge set of $I(d)$ is the intersection of all edge sets of realizations of $d$, and $U(d)$ is the union of all edge sets of realizations, and realizations have the property that vertex $i$ has degree $d_i$ for all $i \in [n]$. Given a degree sequence $d$ and a realization $G$ of $d$, observe that $I(d)=U(d)=G$ if and only if $G$ is the unique realization of $d$; by definition this happens if and only if $G$ is a threshold graph. As we will see in Theorem~\ref{thm: envelopes are thresholds}, threshold graphs have a more general connection to envelope graphs of degree sequences.

Before proceeding we need a few basic definitions and results. An \emph{alternating 4-cycle} in a graph $G$ is a configuration involving four vertices $a,b,c,d$ of $G$ such that $ab,cd$ are edges of $G$ and neither $ad$ nor $bc$ is an edge. Observe that if $G$ has such an alternating 4-cycle, then deleting $ab$ and $cd$ from $G$ and adding edges $bc$ and $ad$ creates another graph in which every vertex has the same degree as it previously had in $G$. It follows that none of the pairs $\{a,b\}$, $\{b,c\}$, $\{c,d\}$, $\{a,d\}$ is forced in $G$. By a well known result of Fulkerson, Hoffman, and McAndrew~\cite{FulkersonEtAl65}, a graph $G$ shares its degree sequence with some other realization if and only if $G$ contains an alternating 4-cycle. Thus threshold graphs are precisely those without alternating 4-cycles.

\begin{lem}\label{lem: forcible interval}
Suppose that $d_k \geq d_j$. If $ij$ is a forced edge for $d$, then $ik$ is also a forced edge. If $ik$ is a forced non-edge, then $ij$ is also a forced non-edge.
\end{lem}
\begin{proof}
Suppose that $ij$ is a forced edge. If $ik$ is not a forced edge, then let $G$ be a realization of $d$ where $ik$ is not an edge. Since $d_k \geq d_j$ and $j$ has a neighbor (namely $i$) that $k$ does not, $k$ must be adjacent to a vertex $\ell$ to which $j$ is not. However, then there is an alternating 4-cycle with vertices $i,j,k,\ell$ that contains the edge $ij$, a contradiction, since $ij$ is a forced edge. By considering complementary graphs and sequences, this argument also shows that if $ik$ is a forced non-edge, then $ij$ is a forced non-edge as well.
\end{proof}

\begin{thm} \label{thm: envelopes are thresholds}
For any degree sequence $d$, both $I(d)$ and $U(d)$ are threshold graphs.
\end{thm}
\begin{proof}
If $I(d)$ is not a threshold graph, then it contains an alternating 4-cycle with edges we denote by $pq, rs$ and non-edges $qr, ps$. Without loss of generality we may suppose that $p$ has the smallest among the degrees of these four vertices. Since $q$ is forcibly adjacent to $p$, by Lemma~\ref{lem: forcible interval} it must be forcibly adjacent to $r$, a contradiction, since $qr$ is not an edge in $I(d)$.

Similarly, if $U(d)$ has an alternating 4-cycle on $p,q,r,s$ as above, and if we assume that $p$ has the largest degree of these vertices, then by Lemma~\ref{lem: forcible interval} since $s$ is forcibly nonadjacent to $p$ it must be forcibly nonadjacent to $r$, a contradiction.
\end{proof}

We now turn to a precise description of $I(d)$ and $U(d)$. Examining the four scenarios in Theorem~\ref{thm: forced via EG diff} under which forcible adjacency relationships may occur, we notice that if for some $k$ we have $\Delta_k=0$, then
\begin{itemize}
\item the set $B=\{i:1 \leq i \leq k\}$ is a clique in which all pairs of vertices are forcibly adjacent;
\item the set $A=\{i:i>k \text{ and } d_i<k\}$ is an independent set in which all pairs of vertices are forcibly nonadjacent; and
\item each vertex in $C=\{i:i>k \text{ and } d_i \geq k\}$ belongs to a forced edge with each vertex in $B$ and belongs to a forced non-edge with each vertex in $A$.
\end{itemize}

This structure of adjacencies within and between $A$, $B$, and $C$ has arisen many times in the literature. In particular, R.I.~Tyshkevich and others described a graph decomposition based upon it, which we now briefly recall. Our presentation is adapted from \cite{Tyshkevich00}, which contains a more detailed presentation and references to earlier papers.

A \emph{split graph} is a graph $G$ for which there exist disjoint sets $A,B$ such that $A$ is an independent set and $B$ is a clique in $G$, and $V(G)=A \cup B$. We define an operation $\circ$ with two inputs. The first input is a split graph $F$ with a given partition of its vertex set into an independent set $A$ and a clique $B$ (denote this by $(F,A,B)$), and the second is an arbitrary graph $H$. The \emph{composition $(F,A,B)\circ H$} is defined as the graph resulting from adding to the disjoint union $F+H$ all edges having an endpoint in each of $B$ and $V(H)$. For example, if we take the composition of the 5-vertex split graph with degree sequence $(3,2,1,1,1)$ (with the unique partition of its vertex set into a clique and an independent set) and the graph $2K_2$, then the result is the graph on the right in Figure~\ref{fig: composition}.%
\begin{figure}
\centering
\begin{tabular}{m{4cm}m{1cm}m{2.5cm}}
\includegraphics[height=1in]{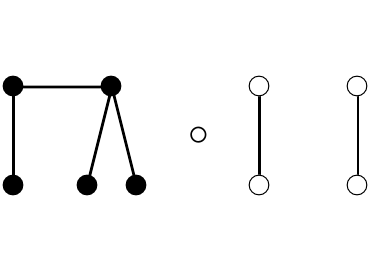} & 
$=$ & 
\includegraphics[height=1in]{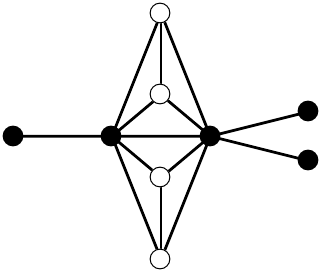}
\end{tabular}
\caption{The composition of a split graph and a graph.}
\label{fig: composition}
\end{figure}

If $G$ contains nonempty induced subgraphs $H$ and $F$ and a partition $A,B$ of $V(F)$ such that $G = (F,A,B) \circ H$, then G is \emph{(canonically) decomposable}; otherwise $G$ is \emph{indecomposable}. Tyshkevich showed in~\cite{Tyshkevich00} that each graph can be expressed as a composition $(G_k,A_k,B_k) \circ \dots \circ (G_1,A_1,B_1) \circ G_0$ of indecomposable induced subgraphs (note that $\circ$ is associative); indecomposable graphs are those for which $k = 0$. This representation is known as the \emph{canonical decomposition} of the graph and is unique up to isomorphism of the indecomposable (partitioned) subgraphs involved.

As observed by Tyshkevich~\cite{Tyshkevich00}, the canonical decomposition corresponds in a natural way with a decomposition of degree sequences of graphs, and it is possible from the degree sequence to deduce whether a graph is canonically indecomposable. In~\cite{HeredUniII}, the author made explicit some relationships between the canonical decomosition of degree sequences and the Erdos--Gallai inequalities recalled in Section~\ref{sec: degree conditions}.

Let $EG(d)$ be the list of nonnegative integers $\ell$ for which $\Delta_\ell=0$, ordered from smallest to largest. We adopt the convention that empty sums have a value of zero in the definitions of $\LHS_\ell(d)$ and $\RHS_\ell(d)$; thus $\Delta_0(d)=0$ for all $d$, and $EG(d)$ always begins with $0$.

\begin{thm}[\cite{HammerSimeone81, Tyshkevich80, TyshkevichEtAl81}] \label{thm: split seqs}
A graph $G$ with degree sequence $d=(d_1,\dots,d_n)$ is split if and only if $m(d)$ is a term of $EG(d)$.
\end{thm}

\begin{thm}[\cite{HeredUniII}, Theorem~5.6] \label{thm: EG and canon decomp}
Let $G$ be a graph with degree sequence $d=(d_1,\dots,d_n)$ and vertex set $[n]$. Suppose that $(G_k,A_k,B_k) \circ \dots \circ (G_1,A_1,B_1) \circ G_0$ is the canonical decomposition of $G$, where $A_0$ and $B_0$ partition $V(G_0)$ into an independent set and a clique, respectively, if $G_0$ is split. A nonempty set $W \subseteq V(G)$ is equal to the clique $B_j$ in the canonical component $G_j$ if and only if $W=\{\ell : t<\ell \leq t'\}$ for a pair $t,t'$ of consecutive terms in $EG(d)$. In this case the corresponding independent set $A_j$ is precisely the set $\{\ell \in [n]: t<d_\ell<t'\}$. Given a term $t$ of $EG(d)$, if $\ell > t$ and $d_\ell = t$ then the canonical component containing $\ell$ consists of only one vertex.
\end{thm}

Thus the condition $\Delta_k(d)=0$ in Theorem~\ref{thm: forced via EG diff} is intimately related to the composition operation $\circ$ and to the canonical decomposition. More generally, we now show that $\Delta_k(d)$ actually measures how far a realization of $d$ is from being a composition of the form described earlier, with slightly relaxed definitions of the sets $A$, $B$, and $C$. Given a subset $S$ of a vertex set of a graph, let $e(S)$ denote the number of edges in the graph having both endpoints in $S$, and let $\overline{e}(S)$ be the number of pairs of nonadjacent vertices in $S$. Given another vertex subset $T$, disjoint from $S$, let $e(S,T)$ denote the number of edges having an endpoint both in $S$ and in $T$, and let $\overline{e}(S,T)$ denote the number of pairs of nonadjacent vertices containing a vertex from each of $S$ and $T$.

\begin{lem}\label{lem: EG diff counts this}
Let $G$ be an arbitrary realization of a degree sequence $d=(d_1,\dots,d_n)$. Given fixed $k \in [n]$, let $B=\{i:1 \leq i \leq k\}$, and let $A$ and $C$ be disjoint sets such that $A \cup C = V(G)-B$, each vertex in $A$ has degree at most $k$, and each vertex in $C$ has degree at least $k$.  

The $k$th Erd\H{o}s--Gallai difference is given by \[\Delta_k(d) = 2e(A) + 2\overline{e}(B) + e(A,C) + \overline{e}(B,C).\]
\end{lem}
\begin{proof}
Observe that summing the degrees in $B$ yields $2e(B)+e(A,B)+e(B,C)$, and a similar statement holds for $A$. Then 
\begin{align*}
\Delta_k(d) &= k(k-1) + \sum_{\ell > k}\min\{k,d_\ell\} - \sum_{\ell \leq k} d_\ell\\
&= k(k-1) + \sum_{\ell \in A} d_\ell + k|C| - (2e(B)+e(A,B)+e(B,C))\\
&= 2\left(\binom{k}{2} - e(B)\right) + 2e(A) + e(A,C) + (|B||C| - e(B,C)),
\end{align*}
and the claim follows.
\end{proof}

Observe that Lemma~\ref{lem: EG diff counts this}, besides providing the corollary below, gives another illustration of the role that a value of $0$ or $1$ for $\Delta_k (d)$ has in producing forced edges and non-edges (as in Theorem~\ref{thm: forced via EG diff}) and in forcing the canonical decomposition structure (as in Theorem~\ref{thm: EG and canon decomp}).

\begin{cor} \label{cor: EG diff at least 2}
Let $d=(d_1,\dots,d_n)$ be a degree sequence. For all $k > m(d)$, we have $\Delta_k(d) \geq 2$.
\end{cor}
\begin{proof}
Since $k>m(d)$, we know that $d_k<k-1$, so any set $B$ of $k$ vertices of highest degree in a realization of $d$ cannot form a clique; thus $\Delta_k(d) \geq 2$ by Lemma~\ref{lem: EG diff counts this}.
\end{proof}

We now use our results in Section~\ref{sec: degree conditions} to determine $I(d)$ and $U(d)$. We begin with a quick observation and some definitions we will use throughout the theorem and its proof. 

\begin{obs}\label{obs: two vtcs}
If an indecomposable canonical component $(G_i,A_i,B_i)$ has more than one vertex, then both $A_i$ and $B_i$ must have at least two vertices.
\end{obs}

Let $G$ be a graph with degree sequence $d=(d_1,\dots,d_n)$ on vertex set $[n]$, and suppose that $G$ has canonical decomposition $(G_k,A_k,B_k)\circ\dots\circ(G_1,A_1,B_1)\circ G_0$.

Let $p$ be the last element of $EG(d)$, and let $q$ be the largest value of $k$ for which $\Delta_k(d) \leq 1$. If $G_0$ is split, let $A_0,B_0$ be a partition of $G_0$ into an independent set and a clique, respectively. If $G_0$ is not split, then define \begin{align*}
B'_0 &= \{i \in [n]: p < i \leq q\};\\
A'_0 &= V(G_0) - B'_0;\\
A''_0 &= \{i \in [n] : i > q \text{ and } p < d_i < q \}; \\
B''_0 &= V(G_0) - A''_0.
\end{align*}

Further let $C_1$ (respectively $C_2$) denote an abstract split canonical component consisting of a single vertex lying in the independent set of the component (in the clique of the component). For $i \in \{1,2\}$ and $j$ a natural number, let $C_i^{j}$ represent the expression $C_i \circ \dots \circ C_i$, where there are $j$ terms $C_i$ in the composition.

\begin{thm}\label{thm: envelope formulas}
Given the graph $G$ with degree sequence $d$, with the canonical components of $G$ and other sets as defined above, 
the graph $I(d)$ is isomorphic to 
\[C_1^{|A_k|} \circ C_2^{|B_k|}\circ \dots \circ C_1^{|A_{1}|} \circ C_2^{|B_{1}|} \circ C_1^{|A_0|}\circ C_2^{|B_0|}\] if $G$ is split, and to 
\[C_1^{|A_k|} \circ C_2^{|B_k|}\circ \dots \circ C_1^{|A_{1}|} \circ C_2^{|B_{1}|} \circ C_1^{|A'_0|}\circ C_2^{|B'_0|}\] otherwise.

Similarly, the graph $U(d)$ is isomorphic to
\[C_2^{|B_k|} \circ C_1^{|A_k|} \circ \dots \circ C_2^{|B_{1}|} \circ C_1^{|A_{1}|} \circ C_2^{|B_0|} \circ C_1^{|A_0|}\] if $G$ is split, and to \[C_2^{|B_k|} \circ C_1^{|A_k|} \circ \dots \circ C_2^{|B_{1}|} \circ C_1^{|A_{1}|} \circ C_2^{|B''_0|}\circ C_1^{|A''_0|}\] otherwise.
\end{thm}
\begin{proof}
By definition, $\Delta_q\leq 1$, and by Theorem~\ref{thm: EG and canon decomp}, it follows that each vertex $i \in [n]$ of $G$ belonging to a set $B_j$ for $j \geq 0$ satisfies $i \leq q$. Theorem~\ref{thm: forced via EG diff}(1) implies that any two vertices in a clique $B_j$ are joined by a forced edge, as are any two vertices in $B'_0$, if $G_0$ is not split.

Consider any two vertices $i,i'$ belonging to the set $A_j$ for $j \geq 0$. By Observation~\ref{obs: two vtcs}, $B_j$ must be nonempty, so it follows from Theorem~\ref{thm: EG and canon decomp} that $d_i<p$ and $d_{i'}<p$. Since $i,i'$ do not belong to $B_\ell$ for any $\ell$, Theorem~\ref{thm: EG and canon decomp} also implies that $i,i' > p$, so by Theorem~\ref{thm: forced via EG diff}(3) the pair $\{i,i'\}$ is a forced non-edge. Similarly, any two vertices in $A''_0$ are forcibly nonadjacent.

Now suppose that vertices $i,i'$ satisfy  $i \in B_j$ and $i' \in V(G_\ell)$, with $\ell < j$. From the adjacencies required by the canonical decomposition we see that $d_{i'}$ is at least as large as $|B_k \cup B_{k-1} \cup \dots \cup B_j|$, and it follows from Theorem~\ref{thm: EG and canon decomp} that this latter number equals a term $t'$ of $EG(d)$ for which $i \leq t'$. By Theorem~\ref{thm: forced via EG diff}(2), the pair $\{i,i'\}$ is a forced edge.

Suppose instead that vertices $i,i'$ satisfy  $i \in A_j$ and $i' \in V(G_\ell)$, with $\ell < j$. Again letting $t'=|B_k \cup B_{k-1} \cup \dots \cup B_j|$, Theorem~\ref{thm: EG and canon decomp} implies that $\Delta_{t'}(d)=0$ and that $i,i'>t'$. The adjacencies of the canonical decomposition imply that $d_{i'} \geq t'$ and that $d_{i} \leq t'$. Theorem~\ref{thm: forced via EG diff}(4) then implies that $\{i,i'\}$ is a forced non-edge.

We now show that all other pairs of vertices in $G$ are unforced, beginning with those within a split canonical component. Suppose that $i \in B_j$ and $i' \in A_j$ for some $j \geq 0$. Any neighbor of $i'$ in $G$ other than $i$ is a neighbor of $i$; furthermore, since $G_j$ is indecomposable, $i'$ has at least one non-neighbor in $B_j$, which must be a neighbor of $i$, so we conclude that $d_i > d_{i'}$ and $i<i'$. Now by Theorem~\ref{thm: EG and canon decomp}, there exist consecutive terms $t$ and $t'$ of $EG(d)$ such that $t < i \leq t'$ and $t < d_{i'} < t'$.

We verify that none of the conditions in Theorem~\ref{thm: forced via EG diff} are satisfied by the pair $\{i,i'\}$. First, since $G_j$ is indecomposable, vertex $t'$ must have at least one neighbor in $A_j$, so $d_{t'} \geq t'$. Thus $i'>t'$, and since $d_{i'}<t'$, we see that $\{1,\dots,i'\}$ is not a clique, so by Lemma~\ref{lem: EG diff counts this} we see that $\Delta_\ell(d) \geq 2$ for all $\ell \geq i'$. Thus condition (1) of Theorem~\ref{thm: forced via EG diff} does not apply to the pair $\{i,i'\}$.

Condition (2) does not apply, since $t'$ is the smallest term of $EG(d)$ at least as large as $i$, and $d_{i'}<t'$. Condition (3) likewise cannot apply, since $\{1,\dots,t'\}$ is a clique and hence $d_i \geq i-1$. Finally, since $t'$ is the smallest term of $EG(d)$ at least as large as $d_{i'}$, but $i \leq t'$, condition (4) does not apply.

It remains to show that $\{i,i'\}$ is unforced if $G_0$ is not split and vertices $i,i' \in V(G_0)$ don't both belong to $B'_0$ or both belong to $A''_0$. Assume that $i<i'$.

If at least one of $i,i'$ does not belong to $B'_0$, then we claim that $\{i,i'\}$ cannot be a forced edge. Indeed, note that $i'>q$ and $i>p$, so neither of conditions (1) or (2) of Theorem~\ref{thm: forced via EG diff} applies. 

If at least one of $i,i'$ does not belong to $A''_0$, then we claim that $\{i,i'\}$ is not a forced non-edge. Indeed, note that $d_i \geq d_{i'}$, and since $G_0$ is indecomposable and has more than one vertex, we have $d_{i'} > p$; this implies that $\{i,i'\}$ fails condition (4). We also see that $i \leq q$ or $d_i \geq q$; in either case condition (3) does not apply.

Having characterized all pairs of vertices as forced or unforced, we can now summarize the structure of $I(d)$ and $U(d)$. If we form a correspondence between each vertex in $A_j$ (respectively, in $A'_0$, in $A''_0$, in $B_j$, in $B'_0$, in $B''_0$) with a vertex of $C_1^{|A_j|}$ (of $C_1^{|A'_0|}$, of $C_1^{|A''_0|}$, of $C_2^{|B_j|}$, of $C_2^{|B'_0|}$, of $C_2^{|B''_0|}$) in the claimed expressions for $I(d)$ and $U(d)$, the correspondence naturally leads to an exact correspondence between the edges in either of the first two expressions and the edges in $I(d)$. Likewise, the edges in the third and fourth expressions in the theorem statement correspond precisely to the edges in $U(d)$.
\end{proof}

A well known and useful characterization of threshold graphs (see~\cite[Theorem 1.2.4]{MahadevPeled95}) states that a graph is threshold if and only if it can be constructed from a single vertex by iteratively adding dominating and/or isolated vertices. The expressions in Theorem~\ref{thm: envelope formulas} describe how the envelope graphs of $d$ can be constructed in this way: because of the requirements of the operation $\circ$, as we read from right to left, a term $C_1^a$ corresponds to adding $a$ isolated vertices in sequence, and a term $C_2^b$ corresponds to adding $b$ dominating vertices.

\begin{exa}
If $d$ is the degree sequence of the graph on the right in Figure~\ref{fig: composition}, then $I(d)$ is formed by starting with a single vertex, adding three more isolated vertices in turn, adding two dominating vertices, and finishing with three more isolated vertices. The graph $U(d)$ is formed by starting with a single vertex, adding three more dominating vertices, three more isolated vertices, and then two more dominating vertices.
\end{exa}

Note that if $d$ is threshold, then $I(d)=U(d)=d$ (as expected), because every canonical component $(G_i,A_i,B_i)$ of a threshold graph contains only a single vertex (we see this from the dominating/isolated vertex construction described above), so for all $i$ either $A_i$ or $B_i$ is empty, and the expressions in Theorem~\ref{thm: envelope formulas} simplify to return the canonical decomposition of the unique realization of $d$.

\section{Threshold graphs and the dominance order}

In this section we compare the forcible adjacency relationships of degree sequences that are comparable under the dominance order.

Given lists $a=(a_1,\dots,a_k)$ and $b=(b_1,\dots,b_\ell)$ of nonnegative integers, with $a_1 \geq \dots \geq a_k$ and $b_1 \geq \dots \geq b_\ell$, we say that $a$ \emph{majorizes} $b$ and write $a \succeq b$ if $\sum_{i=1}^k a_i = \sum_{i=1}^\ell b_i$ and for each $j \in \{1,\dots,\min(k,\ell)\}$ we have $\sum_{i=1}^j a_i \geq \sum_{i=1}^j b_i$. The partial order induced by $\succeq$ on lists of nonnegative integers with a fixed sum $s$ and length $n$ is called the \emph{dominance} (or \emph{majorization}) \emph{order}, and we denote the associated partially ordered set by $\mathcal{P}_{s,n}$.

(We remark that requiring sequences to have the same length and allowing terms to equal 0 are slight departures from how the dominance poset is often described. We do so here for convenience in the statements of results below.)

The dominance order plays an important role in the study of graphic lists. It is known that if for $a,b \in \mathcal{P}_{s,n}$ it is true that $a$ is graphic and $a \succeq b$, then $b$ is also graphic; thus the degree sequences form an ideal in $\mathcal{P}_{s,n}$. The maximal graphic lists are precisely the threshold sequences~\cite{PeledSrinivasan89}.

We define a \emph{unit transformation} on a nonincreasing integer sequence to be the act of decreasing a sequence term by 1 and increasing an earlier term by 1 while maintaining the descending order of terms. This operation is best illustrated by the Ferrers diagram of the sequences, where sequence terms are depicted by left-justified rows of dots. Note that if $a$ results from a unit tranformation on $b$, then the Ferrers diagram of $a$ differs from that of $b$ by the removal of a dot from one row of $b$ to a row higher up in the diagram. In Figure~\ref{fig: elem trans}, the second and third sequences each result from a unit transformation on the first sequence.
\begin{figure}
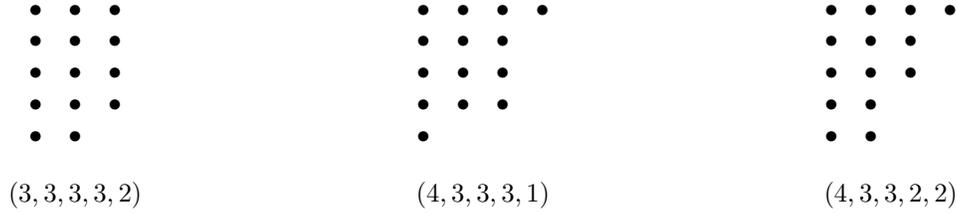

\centering
\begin{minipage}{0.3\textwidth}
\centering
\[\begin{array}{ccc}
\bullet & \bullet & \bullet\\
\bullet & \bullet & \bullet\\
\bullet & \bullet & \bullet\\
\bullet & \bullet & \bullet\\
\bullet & \bullet & \\
\end{array}\]
\[(3,3,3,3,2)\]
\end{minipage}
\quad
\begin{minipage}{0.3\textwidth}
\centering
\[\begin{array}{cccc}
\bullet & \bullet & \bullet & \bullet\\
\bullet & \bullet & \bullet & \\
\bullet & \bullet & \bullet & \\
\bullet & \bullet & \bullet & \\
\bullet &  &  & \\
\end{array}\]
\[(4,3,3,3,1)\]
\end{minipage}
\quad
\begin{minipage}{0.3\textwidth}
\centering
\[\begin{array}{cccc}
\bullet & \bullet & \bullet & \bullet\\
\bullet & \bullet & \bullet & \\
\bullet & \bullet & \bullet & \\
\bullet & \bullet &  & \\
\bullet & \bullet &  & \\
\end{array}\]
\[(4,3,3,2,2)\]
\end{minipage}
\caption{Ferrers diagrams depicting elementary transformations}\label{fig: elem trans}
\end{figure}

A fundamental lemma due to Muirhead~\cite{Muirhead03} says that $a \succeq b$ if and only if $a$ may be obtained by performing a sequence of unit transformations on $b$.

\begin{thm} \label{thm: majorization preserves forced}
Let $d$ and $e$ be graphic elements of $\mathcal{P}_{s,n}$. If $d \succeq e$ and $\{i,j\}$ is a forced pair for $e$, then $\{i,j\}$ is a forced pair for $d$.
\end{thm}
\begin{proof}
We may obtain $d$ from a sequence of unit transformations on the sequence $e$. The intermediate sequences resulting from these individual transformations all majorize $e$, so it suffices to assume that $d$ can be obtained from just one unit transformation. In other words, we assume that there exist indices $s$ and $t$ such that $s<t$ and \[ d_{\ell} = \begin{cases}
e_\ell+1 & \ell = s\\
e_\ell-1 & \ell = t\\
e_\ell & \text{otherwise}
\end{cases}.
\]
Suppose now that $\{i,j\}$ is a forced edge for $e$. By Theorem~\ref{thm: forced iff not graphic} $e^+(i,j)$ is not graphic, so by Theorem~\ref{thm: ErdosGallai} there exists an index $k$ such that $k \geq i$ and \[\sum_{\ell \leq k} e^+(i,j)_\ell > k(k-1) + \sum_{\ell>k} \min\{k, e^+(i,j)_\ell\}.\] Since the actions of increasing two terms of a sequence and performing a unit transformation on a sequence together yield the same result regardless of the order in which they are carried out, we have
\begin{align*}
\sum_{\ell \leq k} d^+(i,j)_\ell &\geq \sum_{\ell \leq k} e^+(i,j)_\ell\\
&> k(k-1) + \sum_{\ell>k} \min\{k, e^+(i,j)_\ell\}\\
&\geq k(k-1) + \sum_{\ell>k} \min\{k, d^+(i,j)_\ell\}.
\end{align*}
Thus $d^+(i,j)$ is not graphic, and by Theorem~\ref{thm: forced iff not graphic} $\{i,j\}$ is a forced edge for $d$.

A similar argument holds if $\{i,j\}$ is a forced non-edge for $e$, making $\{i,j\}$ a forced non-edge for $d$.
\end{proof}

\begin{exa}
The degree sequence $(2,1,1,1,1)$ is majorized by $(2,2,1,1,0)$, which is in turn majorized by $(3,1,1,1,0)$. The first sequence has has no forced pair. In the second sequence vertex 5 is forcibly nonadjacent to all other vertices, $\{3,4\}$ is a forced non-edge, and $\{1,2\}$ is a forced edge. These relationships are all preserved in $(3,1,1,1,0)$, and every other pair of vertices is forced as well, since $(3,1,1,1,0)$ is a threshold sequence.
\end{exa}

As illustrated in the previous example, forcible adjacency relationships come into existence as we progress upward in the dominance order, and they persist until the threshold sequences are reached, where every pair of vertices is a forced pair. Thus the proportion of all vertex pairs that are forced may be considered a measure of how close a degree sequence is to being a threshold sequence.

Our results now yield a consequence of Merris~\cite[Lemma 3.3]{Merris03}. We call a degree sequence \emph{split} if it has a realization that is a split graph.

\begin{cor}\label{cor: splits upward closed}
Let $d$ and $e$ be graphic elements of $\mathcal{P}_{s,n}$. If $d \succeq e$ and $e$ is split, then $d$ is split.
\end{cor}
\begin{proof}
As usual, let $G$ be a realization of $e$ with vertex set $[n]$. Since $e$ is split, by Theorem~\ref{thm: split seqs} we know that $\Delta_{m(e)} = 0$. By Theorem~\ref{thm: forced via EG diff} every pair of vertices from $\{1,\dots,m(e)\}$ forms a forced edge. Likewise, any pair of vertices from $\{m(e)+1,\dots,n\}$ forms a forced non-edge. By Theorem~\ref{thm: majorization preserves forced}, these forcible adjacency relationships exist also for $d$, so $\{1,\dots,m(e)\}, \{m(e)+1,\dots,n\}$ is a partition of the vertex set of any realization of $d$ into a clique and an independent set; hence $d$ is also split.
\end{proof}

Note that by Theorems~\ref{thm: forced via EG diff} and~\ref{thm: EG and canon decomp}, adjacencies and non-adjacencies between vertices in distinct canonical components, as well as adjacencies between two clique vertices and non-adjacencies between two independent-set vertices in split canonical components, are all forcible adjacency relationships. Thus every realization of the degree sequence of a canonically decomposable graph is canonically decomposable. It is natural to then, as we did for split sequences, refer to a degree sequence itself as \emph{decomposable} if it has a decomposable realization.

The forcible adjacency relationships between canonical components and inside split components further imply, via an argument similar to that of Corollary~\ref{cor: splits upward closed},  that canonically decoposable graphs have the same majorization property that split graphs do.

\begin{cor}
Let  $d$ and $e$ be graphic elements of $\mathcal{P}_{s,n}$. If $d \succeq e$ and $e$ is canonically decomposable, then $d$ is canonically decomposable.
\end{cor}

More generally, all sequences with at least one forced pair form an upward-closed set in $\mathcal{P}_{s,n}$. We now show, in fact that these sequences lie close in $\mathcal{P}_{s,n}$ to split or decomposable sequences. The key will be the observation that according to Lemma~\ref{lem: EG diff counts this}, having a small Erd\H{o}s--Gallai difference requires a graph to have a vertex partition that closely resembles that of a split or decomposable graph.

Our measurement of ``closeness'' in $\mathcal{P}_{s,n}$ will involve counting covering relationships. A unit transformation on a nonincreasing integer sequence is said to be an \emph{elementary transformation} if there is no longer sequence of unit transformations that produces the same result; in other words, an elementary transformation changes an integer sequence into one that immediately covers it in $\mathcal{P}_{s,n}$. As shown by Brylawski~\cite{Brylawski73}, a unit transformation on a sequence $b=(b_1,\dots,b_\ell)$ is an elementary transformation if and only if, supposing that the $p$th term of $b$ is increased and the $q$th term is decreased, we have either $q=p+1$ or $b_p=b_q$. The rightmost sequence in Figure~\ref{fig: elem trans} shows the result of an elementary transformation on the original sequence, while the middle sequence shows a non-elementary unit transformation.

\begin{thm}\label{thm: three steps}
If $e$ is a graphic sequence in $\mathcal{P}_{s,n}$ that induces any forcible adjacency relationships among the vertices of its realizations, then some sequence $d$ that is split or canonically decomposable is located at most three elementary transformations above $e$ in $\mathcal{P}_{s,n}$.
\end{thm}
\begin{proof}
By Theorem~\ref{thm: forced via EG diff}, $e$ can only force vertices to be adjacent or nonadjacent if $\Delta_k(e) \leq 1$ for some positive $k$. If for such a $k$ we have $\Delta_k(e)=0$, then by Theorems~\ref{thm: split seqs} and~\ref{thm: EG and canon decomp} we may let $d=e$.

Suppose instead that $\Delta_k(e)=1$ for some $k$, and let $G$ be a realization of $e$ on vertex set $[n]$. Partition $V(G)$ into sets $A$, $B$, and $C$ as in the statement of Lemma~\ref{lem: EG diff counts this}, with $B=\{i: 1 \leq i \leq k\}$. Since $\Delta_k(e)=1$, this lemma implies that $A$ is an independent set, $B$ is a clique, and exactly one of the following cases holds:
\begin{enumerate}
\item[(1)] there is a single edge joining a vertex in $A$ with a vertex in $C$, and all edges possible exist joining vertices in $B$ with vertices in $C$;
\item[(2)] there is a single non-edge between a vertex in $B$ and a vertex in $C$, and there are no edges joining vertices in $A$ with vertices in $C$.
\end{enumerate}
We consider each of these cases in turn. We will use the following statement, which is  proved using elementary edge-switching arguments:

\medskip
\noindent \textsc{Fact}~\cite[Lemma~3.2]{BarrusDonovan16}: \emph{Given a vertex $v$ of a graph $G$ and a set $T \subseteq V(G)-\{v\}$, suppose that $v$ has $p$ neighbors in $T$. For any set $S$ of $p$ vertices of $T$ having the highest degrees in $G$, there exists a graph $G'$ with the same vertex set as $G$ in which the neighborhood of $v$, restricted to $T$, is $S$, all neighbors of $v$ outside of $T$ are the same as they are in $G$, and every vertex has the same degree in $G'$ as in $G$.}

\medskip
In the first case, let $a$ be the vertex of $A$ having a neighbor in $C$. By the fact above we may assume that the neighbor of $a$ in $C$ is a vertex having the highest degree in $G$ among vertices of $C$; call this neighbor $c$. Since the degree of $a$ is at most $k$, there must be some vertex in $B$ to which $a$ is not adjacent; using the fact again, we may assume that this non-neighbor (call it $b$) has the smallest degree in $G$ among vertices of $B$. Now deleting edge $ac$ and adding edge $ab$ produces a graph having a degree sequence $d$ which, using partition $A,B,C$, we see is canonically decomposable.

In the second case, some vertex of $C$ has a non-neighbor in $B$. By the fact above, we may assume that the non-neighbor in $B$ is a vertex $b$ having the lowest degree in $G$ among vertices of $B$. Now since every vertex of $C$ has degree at least $k$ and no neighbors in $A$, but some vertex in $C$ has a non-neighbor in $B$, this vertex in $C$ must have a neighbor in $C$. Using the fact again, we may then assume that the non-neighbor of $b$ in $C$ is a vertex $u$ having smallest degree among the vertices in $C$. Using the fact once again, we may assume that the neighbors of $u$ in $C-\{u\}$ have as high of degree as possible. Now let $c$ be a neighbor of $u$ that has maximum degree among the vertices of $C$. Deleting the edge $uc$ and adding the edge $ub$ produces a graph a graph having degree sequence $d$ for which, using partition $A,B,C$, we see is canonically decomposable.

In both cases, the effect of creating degree sequence $d$ from $e$ was to perform a unit transformation which reduced the largest degree of a vertex in $C$ and increased the smallest degree of a vertex in $B$. Since by assumption the degrees of vertices in $B$ are the highest in the graph, and the degrees of vertices in $C$ may be assumed to precede those of vertices in $A$ in the degree sequence, the creation of $d$ from $e$ is equivalent to a unit transformation on $e$. In fact, it is equivalent to an elementary transformation if $b$ and $c$ are the unique vertices with their respective degrees or if $b$ has the same degree as $c$. Otherwise, we may accomplish this unit transformation using two or three elementary transformations, as follows (we use $\deg(v)$ to denote the degree of a vertex $v$):

If both $\deg(b)$ and $\deg(c)$ appear multiple times in $d$ and $\deg(b) > \deg(c)+1$, then decrease the last term equal to $\deg(b)$ while increasing the first term equal to $\deg(b)$; decrease the last term equal to $\deg(c)$ while increasing the first term equal to $\deg(c)$; and decrease the term currently equal to $\deg(c)+1$ while increasing the term currently equal to $\deg(b)-1$.

If $\deg(b)$ appears multiple times in $d$ and either $\deg(c)$ appears only once or $\deg(b)=\deg(c)+1$, then decrease the last term equal to $\deg(b)$ while increasing the first term equal to $\deg(b)$; then decrease the last term equal to $\deg(c)$ while increasing the first term currently equal to $\deg(b)-1$.

If $\deg(c)$ appears multiple times in $d$ and $\deg(b)$ appears only once, then decrease the last term equal to $\deg(c)$ while increasing the first term equal to $\deg(c)$; then decrease the last term currently equal to $\deg(c)+1$ while increasing the first term currently equal to $\deg(b)$.
\end{proof}

We conclude by showing that the bound in Theorem~\ref{thm: three steps} is sharp for infinitely many degree sequences.

\begin{exa} Consider the sequence $s=\left((15+2j)^{(5)},  6^{(7+2j)}, 3^{(7)}\right)$, where $j$ is any nonnegative integer; note that $\Delta_5(s) = 1$ and that $\Delta_k(s) \neq 0$ for all positive $k$. Let $s'$ and $s''$ denote sequences obtained by performing respectively one and two elementary transformations on $s$.

Observe by inspection that $s''_7 \geq s'_7 \geq s_7 = 6$ and $s''_8 = s'_8 = s_8 < 7$; thus $m(s)=m(s')=m(s'')=7$. To test whether any of $s,s',s''$ has a realization that is decomposable, by Theorem~\ref{thm: EG and canon decomp} and Corollary~\ref{cor: EG diff at least 2}, it suffices to test whether equality holds in any of the first seven Erd\H{o}s--Gallai inequalities for the corresponding sequence.

Recalling our notation from before, we see that since $\LHS_k(s) \leq \LHS_k(s') \leq \LHS_k(s'')$ and $\RHS_k(s) \leq \RHS_k(s') \leq \RHS_k(s'')$, if any of $s,s',s''$ satisfied the $k$th Erd\H{o}s--Gallai inequality with equality, it would follow that $\RHS_k(s) \leq LHS_k(s'')$. Now consider the table below, which shows the maximum possible value for $\LHS_k(s'')$ and the value of $\RHS_k(s)$ for each $k \in \{1,\dots,7\}$.

\begin{center}
\begin{tabular}{ccc}
\hline
$k$ & max $\LHS_k(s'')$ & $\RHS_k(s)$\\ \hline
$1$ & $16+2j$  & $18+2j$  \\ \hline
$2$ & $32+4j$  & $36+4j$  \\ \hline
$3$ & $47+6j$  & $54+6j$  \\ \hline
$4$ & $61+8j$  & $65+8j$  \\ \hline
$5$ & $75+10j$ & $76+10j$ \\ \hline
$6$ & $82+10j$ & $87+12j$ \\ \hline
$7$ & $89+10j$ & $93+12j$ \\ \hline
\end{tabular}
\end{center}

We see that each of $s,s',s''$ is graphic. Furthermore, since in no case does $\RHS_k(s) \leq LHS_k(s'')$, we conclude that any canonically decomposable degree sequence that majorizes $s$ must be separated from $s$ by at least three elementary transformations. (As guaranteed above, the sequence $\left(16+2j, (15+2j)^{(4)}, 6^{(6+2j)},5,3^{(7)}\right)$ is located three elementary transformations above $s$ and is the degree sequence of a canonically decomposable graph.) 
\end{exa}


\begin{thebibliography}{99}
\bibitem{BarrusDonovan16} M.D.~Barrus and E.~Donovan, Neighborhood degree lists of graphs, submitted.

\bibitem{HeredUniII} M.D.~Barrus, Hereditary unigraphs and Erdos--Gallai equalities, Discrete Math., 313 (2013), no.~21, 2469--2481.

\bibitem{Brylawski73} T.~Brylawski, The lattice of integer partitions,  Discrete Math., 6 (1973), no.~3, 201--219.

\bibitem{ChvatalHammer73} V.~Chv\'{a}tal and P.L.~Hammer, Set-packing and threshold graphs, Research Report, Comp.~Sci.~Dept.~University of Waterloo, Canada CORR 73-21 (1973).

\bibitem{ChvatalHammer77} V.~Chv\'{a}tal and P.L.~Hammer, Aggregation of inequalities in integer programming. In P.L.~Hammer, E.L.~Johnson, B.H.~Korte, and G.L.~Nemhauser, editors, Studies in Integer Programming, pages 145--162. North-Holland, New York, 1977. Annals of Discrete Mathematics,  1.

\bibitem{ErdosGallai60} P.~Erd\H{o}s and T.~Gallai, Graphen mit Punkten vorgeschriebenen Grades, Math.~Lapok., 11 (1960), 264--272.

\bibitem{FulkersonEtAl65} D.R.~Fulkerson, A.J.~Hoffman, and M.H.~McAndrew, Some properties of graphs with multiple edges, Can.~J.~Math., 17 (1965), 166–177.

\bibitem{HammerIbarakiSimeone78} P.L.~Hammer, T.~Ibaraki, and B.~Simeone, Degree sequences of threshold graphs, Congres.~Numer.~21 (1978) 329--355.

\bibitem{HammerIbarakiSimeone81} P.L.~Hammer, T.~Ibaraki, and B.~Simeone, Threshold sequences, SIAM J.~Algebraic Discrete Methods~2 (1981) 39--49.

\bibitem{HammerSimeone81} P.L.~Hammer and B.~Simeone, The splittance of a graph, Combinatorica
1 (1981) 275--284.

\bibitem{MahadevPeled95} N.V.R.~Mahadev and U.N.~Peled, Threshold graphs and related topics. Annals of Discrete Mathematics, 56. North-Holland Publishing Co., Amsterdam, 1995. 

\bibitem{Merris03} R.~Merris, Split graphs, European J.~Combin.~24 (2003), no.~4, 413--430.

\bibitem{Muirhead03} R.F.~Muirhead, Some methods applicable to identities and inequalities of symmetric algebraic functions on $n$ letters, Proceedings of the Edinburgh Mathematical Society, 21 (1903), 144--157.

\bibitem{PeledSrinivasan89} U.N.~Peled and M.K.~Srinivasan, The polytope of degree sequences, Linear Algebra Appl., 114/115 (1989), 349--377.

\bibitem{Tyshkevich80} R.I.~Tyshkevich, The canonical decomposition of a graph, Doklady Akad.~Nauk, BSSR 24 (1980) 677--679. In Russian.

\bibitem{Tyshkevich00} R.~Tyshkevich, Decomposition of graphical sequences and unigraphs, Discrete Math.~220 (2000), no.~1--3, 201--238.

\bibitem{TyshkevichEtAl81} R.I.~Tyshkevich, O.I.~Melnikow, and V.M.~Kotov, On graphs and degree sequences: the canonical decomposition, Kibernetica 6 (1981) 5--8. In Russian.
\end{thebibliography}
\end{document}